\documentclass[12pt, reqno]{amsart}
\usepackage{amsmath, amsthm, amscd, amsfonts, amssymb, graphicx, color,verbatim}
\usepackage[bookmarksnumbered, colorlinks, plainpages]{hyperref}
\hypersetup{colorlinks=true,linkcolor=red, anchorcolor=green, citecolor=cyan, urlcolor=red,
	filecolor=magenta, pdftoolbar=true}

\newtheorem{theorem}{Theorem}[section]

\newtheorem {question}[theorem]{Question}
\newtheorem{proposition}[theorem]{Proposition}
\newtheorem{corollary}[theorem]{Corollary}
\theoremstyle{definition}
\newtheorem{definition}[theorem]{Definition}
\newtheorem{example}[theorem]{Example}

\newtheorem{remark}[theorem]{Remark}
\newtheorem{\thetheorem}[theorem]{Theorem A}
\numberwithin{equation}{section}

\begin{document}
	\setcounter{page}{1}
	
	\title[Singular inner eigenfunctions of composition operators]{Singular inner eigenfunctions of composition operators}
	\author[V. A. Anjali, P. Muthukumar \MakeLowercase{and} P. Shankar]{V. A. Anjali, P. Muthukumar \MakeLowercase{and} P. Shankar}
	

	\address{V. A. Anjali, Department of Mathematics, Cochin University of Science And Technology,
		Ernakulam, Kerala- 682022, India. }
	\email{\textcolor[rgb]{0.00,0.00,0.84}{anjaliva6446@gmail.com}}
	
	\address{P. Muthukumar, Department of Mathematics, Indian Institute of Technology,
		Kanpur- 208016, India. }
	\email{\textcolor[rgb]{0.00,0.00,0.84}{pmuthumaths@gmail.com, muthu@iitk.ac.in}}
	
	\address{P. Shankar, Department of Mathematics, Cochin University of Science And Technology,
		Ernakulam, Kerala- 682022, India.}
	\email{\textcolor[rgb]{0.00,0.00,0.84}{shankarsupy@gmail.com, shankarsupy@cusat.ac.in}}
	
	
	
	\date{\today
	}

	\begin{abstract}
		 This paper characterizes all the singular inner eigenfunctions of the composition operators $C_\phi$ 
		that arise from discrete measures, when $\phi$  
		is an automorphism of unit disk. By establishing a connection between Beurling and model invariant subspaces, we classify all the inner functions so that the  corresponding Beurling subspace is invariant under the composition operators induced by non-elliptic automorphisms. This classification involves solving the eigenfunction equation for the composition operator. Further, we present some applications of the above-mentioned connection.
		
	\end{abstract}
	\maketitle

\section{Introduction}
The Hardy–Hilbert space on the open unit disc $\mathbb{D}$, denoted by $H^2({\mathbb{D}})$ or simply $H^2$, is the Hilbert space of  all holomorphic functions $f:\mathbb{D}\rightarrow \mathbb{C}$
such that 
$$
\Vert f\Vert_{2}:=\sup\limits_{0\leq r<1}\left(\dfrac{1}{2\pi}\int\limits_{0}^{2\pi}
|f(re^{i\theta})|^2d\theta\right)^{\frac{1}{2}}
$$ 
is finite.
 For a given holomorphic self-map $\phi$ of $\mathbb{D}$, the operator $C_\phi$
defined as $C_\phi(f)=f \circ \phi\,$,  is called the composition operator with the symbol $\phi$. By Littlewood’s subordination theorem \cite[Page~11]{shapiro1}, $C_\phi$
is a bounded linear operator on $H^2$. The theory of composition operator is largely devoted to understanding how the analytic behavior of the symbol $\phi$ is reflected in the operator-theoretic features of $C_\phi$.

The functional equation
\begin{equation*}
	f(\phi(z))=\lambda f(z),~~~\qquad \qquad z \in \mathbb{D},
\end{equation*}
where the holomorphic self-map $\phi$ of $\mathbb{D}$ is given, the holomorphic map $f$ on $\mathbb{D}$ and the scalar $\lambda$ are unknown, is known as the Schröder equation. It has been studied since the late nineteenth century, beginning with the work of  Koenigs \cite{konigs} in 1884 on holomorphic self-map $\phi$ of $\mathbb{D}$ with an interior fixed point.

Schröder’s equation is widely used in complex dynamics, differential equations, and operator theory. This equation is essential for studying the spectral properties of linear operators and the dynamics of holomorphic functions in one and several complex variables.
Gallardo-Guti\'{e}rrez et al. \cite{orbits} gave a constructive characterization of the eigenfunctions of $C_\phi$ on Hardy spaces.
Fricain et al. \cite{modelsubspace} characterized all the Blaschke eigenfunctions of $ C_{\phi}$ induced by non-elliptic automorphisms $\phi$ of $\mathbb{D}$. For further reading on this direction, we refer \cite{cowenit,mobius}.

The classical Invariant Subspace Problem (ISP) asks whether every bounded linear operator on an infinite-dimensional complex separable Hilbert space has a nontrivial proper closed invariant subspace. Nordgren et al. \cite{COMP} obtained an equivalent formulation of the ISP in terms of composition operators on  $H^2$ induced by a hyperbolic automorphism. More recently, Carmo and Noor \cite{noor} provided another reformulation of the ISP via composition operators induced by hyperbolic self-maps of $\mathbb{D}$.

Beurling \cite{beurling} gave a complete description of invariant subspaces of the multiplication operator $M_z$ (the shift operator) on
$H^2$ induced by the coordinate function. He showed that every nontrivial invariant subspace of $M_z$ on $H^2$
is of the form 
$$\theta H^2$$ 
for some inner function 
$\theta$. Motivated by this result, for an inner function 
$\theta$, the subspace $\theta H^2$  is known as a \emph{Beurling subspace}.
Consequently, a closed subspace $M\subseteq H^2$ is invariant under $M_{z}^{*}$ (note that $M$ is also referred as a \emph{model space}) if and only if $M=(\theta H^2)^\perp$ for some inner function $\theta$, where $M_{z}^{*}$ denotes the adjoint of $M_z$, namely, the backward shift operator on $H^2$ defined as 
$$M_{z}^{*}(f)(z)=\dfrac{f(z)-f(0)}{z},~ z\in \mathbb{D}.$$

Mahvidi \cite{mahvidi} studied the common invariant subspaces of two composition operators, along with their lattice containment properties. The complete structure of invariant subspaces for composition operators on $H^2$ induced by parabolic non-automorphisms were established in \cite{parabolic}. The investigation of Beurling subspaces invariant under composition operators was initiated by Chalendar and Partington \cite{Chalendar}. Subsequently, Jones \cite{jones} examined invariant Beurling subspaces of the composition operator $C_\phi$ in the case when $\phi$ is an inner function. Cowen and Wahl \cite{cowen} showed that when $\phi$ has its Denjoy–Wolff point $p$ on the unit circle, the atomic inner function subspaces with a single atom at $p$ are invariant subspaces for $C_\phi$. Matache \cite{valentine} proved that every composition operator on $H^2$ admits a nontrivial invariant Beurling subspace.
 
 Building on these results, Bose et al. \cite{buerlingtype} unified the findings of \cite{cowen, jones, valentine} and provided a characterization for a Beurling subspace $\theta H^2$ to be invariant under $C_\phi$ in terms of the functions $\theta$ and $\phi$. An alternative characterization, expressed in terms of the boundedness of certain weighted composition operators, was later given by Matache \cite{smirnov}. In \cite{article1}, the authors gave a partial answer to a question raised by  Matache \cite{valentine} regarding the singular Beurling  invariant subspaces of composition operators. 
 
   Mashreghi and Sabhankh \cite{finiterank} discussed  the finite-dimensional model invariant subspaces of composition operators. Later, Muthukumar and Sarkar \cite{model} investigated model spaces that are invariant under the composition operator $C_\phi$ on $H^2$. More recently, Muthukumar et al. \cite{modelfinite} presented a complete characterization of finite-dimensional model spaces that are invariant under composition operators. Mashreghi and Sabhankh \cite{javad} characterized all the invariant model subspaces of $C_\phi$, where $\phi$ is an inner function on $\mathbb{D}$. Their characterization depends on the inner eigenfunctions of $C_\phi$.

 In this article, we mainly consider the singular inner eigenfunctions  of $C_\phi$ arising from a discrete measure, when $\phi$ is an automorphism. The rest of the paper is organized as follows. The Section 2 deals with some basic definitions.
In Section 3, we study the inner eigenfunctions of the composition operator $C_\phi$, where $\phi$ is an automorphism of $\mathbb{D}$. In particular, we provide a complete characterization of all singular inner eigenfunctions arising from discrete measures.
In Section 4, we establish a relation between Beurling and model invariant subspaces of composition operators  $C_\phi$ on $H^2$, where the inducing symbol $\phi$ is an automorphism.  This connection enables us to classify all the inner functions $\theta$ such that the Beurling subspace $\theta H^2$ is invariant under the composition operators induced by non-elliptic automorphisms. We  also present some applications of the above-mentioned connection.

\section{Preliminaries}
In this section, we present some notation and the necessary background for what follows.
Let $\mathbb{N}$ denote
the set of all natural numbers and  $\mathbb{Z}$ denote the set of all integers. We denote the open unit disk and unit circle in the complex
plane by $\mathbb{D}$ and $\mathbb{T}$,
respectively.

The algebra of all bounded holomorphic functions on $\mathbb{D}$
with supremum norm is denoted by $H^{\infty}$. A function $\theta\in H^{\infty}$ is said to be an \textit{inner function} if its radial limit $\tilde{\theta}(e^{it})=\lim\limits_{r\rightarrow 1}{\theta}(re^{it})$ satisfies  $|\tilde{\theta}(e^{it})|=1$
a.e. on $\mathbb{T}$. Any inner function $\theta$ can be factorized as
$\theta=BS,$
where $B$ is a Blaschke product, arising from zeros of $\theta$  and $S$ is a non-vanishing inner function (also called as a singular inner function) \cite[Corollary 2.6.6]{texthardy}. It is important to
note that this factorization is
unique up to unimodular constant. For a basic introduction to Blaschke products and singular inner functions, see \cite{duren,texthardy,shapiro1}.
 We recall the following, as we need them in the sequel.
For a given finite positive Borel measure $\mu$ on $\mathbb{T}$, let

\begin{equation}\label{singulardef}
	S_{\mu}(z)= \exp\left(-\int\limits_{\mathbb{T}}\dfrac{t+z}{t-z}d\mu(t)\right)~~~~
	(z\in \mathbb{D}).
\end{equation}
Any singular inner function $S$ is of the form
$S=cS_{\mu}$,
for some finite positive Borel measure $\mu$ on $\mathbb{T}$, which is singular with respect
to Lebesgue measure on $\mathbb{T}$
and $c\in\mathbb{T}$. We say that two singular inner functions $S_1$ and $S_2$ are similar (denoted by $S_1\sim S_2 $), if  $S_1=cS_2$, for some $c \in \mathbb{T}$.
\begin{definition}\label{def}\cite[Definition 7.11 \& 7.12]{measure}
	A measure $\mu $ on 
	$\mathbb{T}$ is  said to be $ singular\, continuous$ (with respect to Lebesgue measure) if $\mu(\{z\})=0$ for all $z \in \mathbb{T}$. Also, $\mu$ is said to be $discrete$ (with respect to Lebesgue measure) if there exists a countable subset $A$ of $\mathbb{T}$ such that $\mu(A^c)=0$.
\end{definition}
Corresponding to a $discrete$  measure $\mu$ on $\mathbb{T}$, there is a sequence $\{p_i:i\in \mathbb{N}$\} (possibly finite), on which the measure is supported. We denote this set by $atom(\mu)$. Thus, $\mu$ can be written as
$$\mu=\sum\limits_{i}a_{i}\delta_{p_i},$$
where $a_i=\mu(p_i)>0$, 
$\delta_{p_i}$ is the Dirac delta measure on $\mathbb{T}$ centered at $p_i$. We may also assume that $\sum_{i}a_i=\mu(\mathbb{T})<\infty$ to make the measure $\mu$ finite. Hence, any singular inner function given by (\ref{singulardef}), corresponding to a finite discrete measure $\mu$ is of the form
$$S_\mu(z)=\exp\Big({\sum\limits_ia_i\frac{z+p_i}{z-p_i}}\Big),$$
where $\{p_i\}_{i \in \mathbb{N} }\subset \mathbb{T}$ with $a_i>0$ and $\sum_{i}a_{i}<\infty$.

For a finite singular measure $\mu$ on $\mathbb{T}$, there exist a unique pair of measures $\mu_a$ and $\mu_c$ on $\mathbb{T}$ such that $\mu_a$ is discrete, $\mu_c$ is singular continuous  and $\mu=\mu_a+\mu_c$ \cite[Proposition 7.13]{measure}. Consequently, any singular inner function $S_\mu$ can be written as
\begin{equation*}{\label{sdecom}}
	S_\mu=S_{\mu_a}S_{\mu_c}.
\end{equation*}
We refer this as the \textit{singular inner factorization} of $S_\mu$.
Hence, any inner function $\theta$ has a unique factorization (where the unimodular constant is absorbed by the Blaschke factor),

\begin{equation*}\label{idecom}
	\theta =BS_\mu= BS_{\mu_a}S_{\mu_c},
\end{equation*} where $B$ is the Blaschke component in the inner factorization  of $\theta$. We refer this as the \textit{inner factorization} of $\theta$.

Based on the position and number of fixed points in $\overline{\mathbb{D}}$, we can classify all automorphisms of $\mathbb{D}$. Any non-identity disk automorphism is:

\begin{enumerate}
	\item \textbf{Elliptic} if it has exactly one fixed point in $\mathbb{D}$.
	\item \textbf{Hyperbolic} if it has exactly two fixed points on $\mathbb{T}$.
	\item \textbf{Parabolic} if it has exactly one fixed point on $\mathbb{T}$.
\end{enumerate}

We conclude this section by recalling the Denjoy–Wolff theorem \cite[Page~78]{shapiro1}. Let $\phi$ be a holomorphic self-map of $\mathbb{D}$. If $\phi$ is not an elliptic automorphism, then there exists a point $p \in \overline{\mathbb{D}}$ such that the iterates $\phi^{[n]}$ (the $n$-fold composition of $\phi$) converge uniformly on compact subsets of $\mathbb{D}$ to the constant function $p$. This point $p$ is called the \emph{Denjoy–Wolff point} of $\phi$.

Throughout the article, $\phi$ denotes the disk automorphism, unless otherwise specified.

\section{Eigenfunctions of composition operators}
In this section, we are primarily interested in  the inner function solutions of the Schröder equation for the composition operators induced by automorphisms of $\mathbb{D}$; that is, to determine the set of all inner functions $\theta$ such that $\theta \circ \phi=\lambda \theta$, for some constant $\lambda$, where $\phi$ is an automorphism of $\mathbb{D}$. 

Let $\theta=BS_\mu$ be the inner factorization. Then we have $C_\phi(\theta H^2)\subseteq \theta H^2$ if and only if $C_\phi(BH^2)\subseteq B H^2$  and $C_\phi(S_\mu H^2)\subseteq S_\mu H^2$ \cite[Corollary 3.11]{article1}. In a similar manner, we could easily prove that, $\theta$ is an eigenfunction of $C_\phi$ if and only if $B$ and $S_\mu$ are eigenfunctions of $C_\phi$. Hence, it suffices to determine all Blaschke and singular inner eigenfunctions separately.

We now proceed to consider the singular inner eigenfunctions. 
The following theorem presents a decomposition property of invariant singular Beurling subspaces and  singular inner eigenfunctions of $C_\phi$, where $\phi$ is an automorphism.
\begin{theorem}\label{splitting}
Let $\phi$ be an automorphism of $\mathbb{D}$ and let $S_\mu= S_{\mu_a}S_{\mu_c}$ be the singular inner factorization of the singular inner function $S_\mu$. Then
\begin{enumerate}
	\item $C_\phi(S_\mu H^2)\subseteq S_\mu H^2$ if and only if  $C_\phi(S_{\mu_a} H^2)\subseteq S_{\mu_a }H^2$ and $C_\phi(S_{\mu_c} H^2)\subseteq S_{\mu_c }H^2$.
	\item $S_\mu$ is an eigenfunction of $C_\phi$ if and only if $S_{\mu_a}$ and $S_{\mu_c}$ are eigenfunctions of $C_\phi$.
\end{enumerate}
\end{theorem}
\begin{proof}
Proof of (1): Assume that $C_\phi(S_\mu H^2)\subseteq S_\mu H^2$. By \cite[Theorem 2.3]{buerlingtype}, we have $
\frac{S_\mu \circ \phi}{S_\mu}=f \in H^\infty$. In fact, $f$ is a singular inner function; say, $f=cS_{\nu_a}S_{\nu_c}$. Since  $S_\mu= S_{\mu_a}S_{\mu_c}$,
we have, $$  (S_{\mu_a}\circ \phi)(S_{\mu_c}\circ \phi)=cS_{\mu_a}S_{\mu_c}S_{\nu_a}S_{\nu_c}.$$ By \cite[Lemma 4]{jones}, 
$$S_{\mu_a}\circ \phi\sim S_{v_1} \text{ and }S_{\mu_c}\circ \phi\sim S_{v_2},$$
where the singular measures $v_1, \,v_2$ are defined as $v_1(E)=\int_{\phi(E)}\frac{1-|\phi(0)|^2}{|t-\phi(0)|^2}d\mu_a(t)$ and $v_2(E)=\int_{\phi(E)}\frac{1-|\phi(0)|^2}{|t-\phi(0)|^2}d\mu_c(t)$, for any Borel subset $E$ of $\mathbb{T}$. From the definition of $v_1$, we can easily see that $supp(v_1)\subseteq \phi^{-1}(supp(\mu_a))$, which makes $supp(v_1)$ countable. Moreover, for any $z \in \mathbb{T}$,
$v_2(\{z\})=\frac{1-|\phi(0)|^2}{|z-\phi(0)|^2}\mu_c(\{\phi(z)\})=0.$
Thus, by the definition, $v_1$ is discrete and ${v_2}$ singular continuous. By the uniqueness of decomposition of singular measure, we get
$$(S_{\mu_a}\circ \phi)\sim S_{\mu_a}S_{\nu_a}
\text{ and }
(S_{\mu_c}\circ \phi)\sim S_{\mu_c}S_{\nu_c}.$$
Hence, \cite[Theorem 2.3]{buerlingtype} implies $C_\phi(S_{\mu_a} H^2)\subseteq S_{\mu_a }H^2$ and $C_\phi(S_{\mu_c} H^2)\subseteq S_{\mu_c }H^2$.

Conversely, assume that $C_\phi(S_{\mu_a} H^2)\subseteq S_{\mu_a }H^2$ and $C_\phi(S_{\mu_c} H^2)\subseteq S_{\mu_c }H^2$. Then, by applying \cite[Theorem 2.3]{buerlingtype}, we obtain $$\dfrac{S_\mu  \circ \phi}{S_\mu}=\Big(\dfrac{S_{\mu_a}\circ \phi}{S_{\mu_a}}\Big)\Big(\dfrac{S_{\mu_c}\circ \phi}{S_{\mu_c}}\Big)\in H^\infty.$$
Thus, $C_\phi(S_\mu H^2)\subseteq S_\mu H^2$.

Proof of (2): Let $S_\mu$ be an eigenfunction of $C_\phi$. Then $S_\mu \circ \phi=\lambda S_\mu$ for some $\lambda \in \mathbb{T}$. Thus, $$  (S_{\mu_a}\circ \phi)(S_{\mu_c}\circ \phi)=\lambda S_{\mu_a}S_{\mu_c}.$$
By using the similar technique given in proof of (1), we get
$$(S_{\mu_a}\circ \phi)=\lambda_1 S_{\mu_a}
\text{ and }
(S_{\mu_c}\circ \phi)=\lambda_2 S_{\mu_c}, $$ for some $\lambda_1, \lambda_2 \in \mathbb{T}$. The converse part is trivial.

\end{proof}
The decomposition discussed above is, in fact, not trivial. Let $\theta=\theta_1 \theta_2$, for some inner functions $\theta, \theta_1,\theta_2$. For an automorphism $\phi$, $\theta$ being an eigenfunction of $C_\phi$ does not guarantee that either of $\theta_1$ and $\theta_2$ be the eigenfunctions of $C_\phi$.
The following example will demonstrate this fact.

\begin{example}
Let $\phi(z)=-z$, for $z \in \mathbb{D}$. For $z \in \mathbb{D}$, let $S_1(z)=e^{\frac{z+1}{z-1}}, S_2(z)=e^{\frac{z-1}{z+1}}$ and  $B_{\alpha}(z):=\frac{\alpha-z}{1-\bar{\alpha}z}$, for $\alpha \in \mathbb{D}$. Define  $\theta=B_{\frac{1}{2}}B_{\frac{-1}{2}}S_1S_2$. We can easily verify that, for any $\alpha \in \mathbb{D}$, $B_{\alpha}\circ \phi=-B_{-\alpha}$. Also $S_1\circ \phi=S_2$ and $S_2\circ \phi=S_1$. Thus  $B_{\frac{1}{2}}B_{\frac{-1}{2}}$ and $S_1S_2$ are eigenfunctions of $C_\phi$. None of other combinations of the individual components of $\theta$ can be an eigenfunction of $C_\phi$. The above observation also extends to the setting of Beurling subspaces.
\end{example}

As a refined version of Theorem \ref{splitting}, using \cite[Corollary 3.11]{article1}, we have the following result:
\begin{corollary}\label{corsplit}
	Let $\phi$ be an automorphism of $\mathbb{D}$ and let $\theta= BS_{\mu_a}S_{\mu_c}$ be the inner factorization of the function $\theta$. Then
	\begin{enumerate}
		\item $C_\phi(\theta H^2)\subseteq \theta H^2$ if and only if  $C_\phi(B H^2)\subseteq BH^2$, $C_\phi(S_{\mu_a} H^2)\subseteq S_{\mu_a }H^2$ and $C_\phi(S_{\mu_c} H^2)\subseteq S_{\mu_c }H^2$.
		\item $\theta$ is an eigenfunction of $C_\phi$ if and only if $B$, $S_{\mu_a}$ and $S_{\mu_c}$ are eigenfunctions of $C_\phi$.
	\end{enumerate}

\end{corollary}

The characterization for a Blaschke Beurling subspace and discrete singular Beurling subspace to be invariant under $C_\phi$, can be found in \cite[Corollary 2.4]{buerlingtype} and \cite[Corollary 2.15]{valentine}, respectively. Based on this, it is natural to propose the following question:
\begin{question}
	Let $\mu$ be a singular continuous measure on $\mathbb{D}$. What will be a measure theoretic characterization for $C_\phi(S_{\mu} H^2)\subseteq S_{\mu }H^2$?
\end{question}

Next, we move to the problem of computing the inner eigenfunctions of the composition operators. To do this, in view of Theorem \ref{splitting}, we now consider the particular case of singular inner eigenfunctions associated with a discrete  measure. The following theorem provides a characterization of discrete  inner eigenfunctions of $C_\phi$. 
\begin{theorem}\label{discrete}
	Let $\phi$ be an automorphism of $\mathbb{D}$ and $\mu$ be a discrete singular measure on $\mathbb{T}$ with $\mu({\mathbb{T}})<\infty$. Then
	$S_\mu$ is an eigenfunction of $C_\phi$ if and only if
	\begin{equation}\label{condition}
		\mu(\{\phi(w)\})=|\phi^{'} (w)|\mu(\{w\}), \,\, \, w \in atom(\mu).
	\end{equation} 
	
\end{theorem}

\begin{proof}
Suppose that  $S_\mu$ is an eigenfunction of $C_\phi$.  Then   $S_\mu$ is also an eigenfunction of $C_{\phi^{-1}}$, which implies $S_\mu\circ \phi^{-1}\sim S_\mu$. 
 By \cite[Lemma 3.6]{orbits}, we have 
\begin{equation}\label{measure}
	\mu(\phi(E))=\int\limits_{E}|\phi’(t)|d\mu(t),
\end{equation}
for all Borel subsets $E$ of $\mathbb{T}$. For an arbitrary $w \in atom(\mu)$, set $E=\{w\}$. By substituting $E$ in the equation (\ref{measure}), we have
$\mu(\{\phi(w)\})=|\phi’(w)|\mu(\{w\})$.
	
	Conversely, assume  the condition (\ref{condition}). Let $E$ be an arbitrary Borel subset of $\mathbb{T}$. Then,
$$\mu(\phi(E))=\sum_{\substack{w\in E\\\phi(w)\in atom(\mu)}}\mu(\{\phi(w)\}).$$
By condition (\ref{condition}), it is clear that $\phi(w)$ is an atom of $\mu$ if and only if $w$ is an atom of $\mu$. Also, it is worth to note that, $\mu(w)=0$, for all $w \notin atom(\mu)$. Thus, again by condition (\ref{condition}), $$\mu(\phi(E))=\sum_{w\in E}|\phi'(w)|\mu(\{w\})=\int\limits_{E}|\phi’(t)|d\mu(t).$$

Therefore, by \cite[Lemma 3.6]{orbits}, $|S_\mu\circ \phi^{-1}(z)|= |S_\mu(z)|$ for all $z \in \mathbb{D}$. Since two singular inner functions with same modulus have same associated measure (see \cite[Page~71]{garnett}), we have $S_\mu\circ \phi^{-1}\sim S_\mu$. Hence, $S_\mu$ is an eigenfunction  of $C_\phi$.
\end{proof} 
For any $w \in \mathbb{T}$, we denote  $Orb_\phi(w)=\{\phi^{[k]}(w):k\in \mathbb{Z}\}$.
\begin{remark}
	The above theorem also asserts that for a finite discrete  measure $\mu$ on $\mathbb{T}$, $S_\mu$ is an eigenfunction of $C_\phi$  if and only if $S_{\mu_{w}}$ is eigenfunction for $C_\phi$ for every $w\in atom(\mu) $ where $atom(\mu_{w})=Orb_\phi(w)$.
\end{remark}
We look into the above result more closely for all the cases of $\phi$, namely when $\phi$ is elliptic, parabolic and hyperbolic automorphism. First, we start with the case that $\phi$ is an elliptic automorphism. Thus, $\phi$ has a unique fixed point in $\mathbb{D}$, say $p$.
     If $\phi'(p)$ is not a primitive root of unity (commonly known as the irrational elliptic case), then by \cite[Theorem 2]{jones}, we have $C_\phi(S_\mu H^2)\nsubseteq S_\mu H^2$, for any singular inner function  $S_\mu$. Thus, $S_\mu$ cannot be an eigenfunction  for $C_\phi$. 
     
     On the other hand, let $\phi$ be an elliptic automorphism fixing $p$ and let $\phi’(p)$ is a primitive n$^{th}$ root of unity  for some $n \in \mathbb{N}$ (known as the rational elliptic case). First we start with a simpler case that $p=0$. For a discrete measure $\mu$ on $\mathbb{T}$, if $ S_{\mu}$ is an eigenfunction of $C_\phi$, then by Theorem \ref{discrete}, we have
     $$\mu(\{\phi(w)\})=\mu(\{w\}) |\phi^{'}(w)|=\mu(\{w\}), w \in \mathbb{T}.$$
     Iteratively, we obtain that $\mu$ is constant on each orbit. This yields the following result. 
     
\begin{theorem}
	Let $\mu$ be a discrete measure, $\phi$ be an rational elliptic automorphism fixing $0$ and let $\phi’(p)$ is a primitive n$^{th}$ root of unity  for some $n \in \mathbb{N}$. Then, $S_\mu$ is an eigenfunction of $C_\phi$  if and only if $\mu$ has the form $$\mu=\sum\limits_{m}a_m\sum\limits_{k=0}^{n-1}\delta_{\phi^{[k]}(w_m)},$$ where $\{a_m\}_{m\in \mathbb{N}}$ is a positive real valued sequence (possibly finite) with $\sum_{m}a_m<\infty$ and $\{w_m\}_{m\in \mathbb{N}}$ is a sequence on unit circle.
\end{theorem}

If $\phi$ is a rational elliptic automorphism fixing a non-zero $p \in \mathbb{D}$, the discrete singular inner eigenfunctions of $C_\phi$ have the same form as given in Theorem \ref{disingular}, where the index set of the summation is finite. The similar proof work for this case as well.

Now, we move to the case that $\phi$ is a non-elliptic automorphism. Since the atom set is a disjoint union of distinct orbits, any singular inner eigenfunction of $C_\phi$, arising from a discrete measure, is a product (possibly finite) of the singular inner functions whose atom set is a single orbit. Thus, it is enough to focus on this elementary model. Consider the measure $\mu$  whose atom set is an orbit of $\phi$ at some element $a \in \mathbb{T}$, i.e., 
 $$atom(\mu)=Orb_{\phi}(a). $$
 As $\phi$ is non-elliptic, the orbit $Orb_{\phi}(a)$ is either singleton or countably infinite, depending on whether $a$ is a fixed point or not.

	Suppose $\phi(a)=a$. Then $atom(\mu)=\{a\}$. By Theorem \ref{discrete},
	$$|\phi’(a)|=\frac{\mu(\{\phi(a)\})}{{\mu(\{a\})}}=1,$$ which forces $\phi$ to be a parabolic automorphism. From lines of the proof of \cite[Lemma 5.4.6]{texthardy}, it is very clear that, for a parabolic automorphism $\phi$ of $\mathbb{D}$,  every unimodular constant is an eigenvalue of $C_\phi$ with an atomic singular inner eigenfunction (singular inner function having singleton atom set).

For $\phi(a)\neq a$, the following theorem gives a characterization for the singular inner eigenfunctions of $C_\phi$ arising from discrete measure with the atom set as an orbit at $a$.

\begin{theorem}\label{disingular}
	Let $\phi$ be a non-elliptic automorphism of $\mathbb{T}$ and let $\mu$ be a discrete measure such that $atom(\mu)=Orb_{\phi}(a)$, for some $a \in \mathbb{T}$ such that $\phi(a)\neq a$. Then $S_\mu$ is an eigenfunction of $C_\phi$ if and only if  $\mu$ is of the form
	$$\mu=\sum\limits_{n \in \mathbb{Z}}a_n\delta_{\phi^{[n]}(a)},$$
	where, $a_0>0$  and for $n \in \mathbb{N}$, 
$$a_n:=a_0\prod\limits_{k=0}^{n-1}|\phi^{'}(\phi^{[k]}(a))|\, \text
{  and  }\, a_{-n}:=\dfrac{a_0}{\prod\limits_{k=1}^{n}|\phi^{’}(\phi^{[-k]}(a))|}.$$
\end{theorem}
\begin{proof}
	Let $\phi$ be a non-elliptic automorphism of $\mathbb{T}$ with $\phi(a)\neq a$, for  $a\in \mathbb{T}$. As a consequence of \cite[Lemma 5.4.4, 5.4.8]{texthardy}, if $n \neq m $, then $\phi^{[n]}(a)\neq \phi^{[m]}(a)$, for $n,m \in \mathbb{Z}$. Let $\mu$ be a discrete measure on $\mathbb{T}$ such that $atom(\mu)=Orb_{\phi}(a)$ and let $a_0:=\mu(\{a\})$. 
	
	Assume that $S_\mu$ is an eigenfunction of $C_\phi$. 
By Theorem \ref{discrete}, we obtain
	$$\mu(\{\phi(a)\})=\mu(\{a\}) |\phi^{'}(a)|.$$
	Inductively, we obtain that for any $n \in \mathbb{N}$,
	\begin{equation*}\label{an}
		a_n:=\mu(\{\phi^{[n]}(a)\})=\mu(\{a\}) \prod\limits_{k=0}^{n-1}|\phi^{’}(\phi^{[k]}(a))|.
	\end{equation*}
and
	\begin{equation*}\label{a-n}
		a_{-n}:=\mu(\{\phi^{[-n]}(a)\})=\dfrac{\mu(\{a\})} {\prod\limits_{k=1}^{n}|\phi^{'}(\phi^{[-k]}(a))|}.
	\end{equation*}

	For the converse, let the measure $\mu$ be defined as in the hypothesis.  In view of the Theorem \ref{discrete}, to prove that $S_\mu$ is an eigenfunction of $C_\phi$, it is enough to prove that $\mu(\mathbb{T})=\sum_{n \in \mathbb{Z}}a_n<\infty.$ 
	For proving the convergence, we consider the hyperbolic and parabolic cases separately.
	
	Let $\phi$ be a hyperbolic automorphism fixing $p$ and $q \in \mathbb{T}$, with the Denjoy-Wolff point as $p$.
	Let $a_n$, $n \in \mathbb{Z}$, be defined as in the hypothesis. Then, we have,

 $$\lim\limits_{n \rightarrow \infty}\frac{a_{n+1}}{a_n}=\lim\limits_{n \rightarrow \infty}|\phi^{’}(\phi^{[n]}(a))|=|\phi^{'}(p)|<1$$ and

  $$\lim\limits_{n \rightarrow \infty}\frac{a_{-n}}{a_{-n+1}}=\lim\limits_{n \rightarrow \infty}\dfrac{1}{|\phi^{'}(\phi^{[-n]}(a))|}=\dfrac{1}{|\phi^{'}(q)|}<1.$$ Thus, by ratio test, the series $\sum_{n \in \mathbb{Z}}a_n$ converges.
	
	Let $\phi$ be a parabolic automorphism.  Without loss of generality, assume that $1$ is the fixed point of $\phi$. Let $\alpha:=\phi^{-1}(0)=re^{it}$. As both the fixed points of $\phi$ are $1$, we get  $\frac{-\alpha(1-\bar{\alpha})}{\bar{\alpha}(1-\alpha)}=1$ (see also \cite[Page ~$78 - 79$]{modelsubspace}). Consequently, $r=\cos t$. Also, note that, for a parabolic automorphism $\phi$ fixing $1$, by \cite[Lemma 5.4.4]{texthardy}, there exist  $b \in \mathbb{R}-\{0\}$, such that 
	$$\phi^{[n]}(z)=\dfrac{2iz+nb(1-z)}{2i+nb(1-z)}= 1-\dfrac{2i(1-z)}{2i+nb(1-z)}, \,\, z \in \mathbb{\overline{D}}.$$
Since $\phi(\alpha)=0$, we obtain $b=\frac{-2i\alpha}{1-\alpha}$. As the ratio test is inconclusive in this case, we use the Rabbe's test to prove the result.
Let $a_n$ be defined as mentioned in the hypothesis. We denote the real part of a complex number $z$ by $Re(z)$.
For $n \in \mathbb{N}$, we have
\begin{equation*}
\begin{split}
	\dfrac{a_n}{a_{n+1}}-1&=\dfrac{1}{|\phi^{'}(\phi^{[n]}(a))|}-1\\
	&=\dfrac{|1-\bar{\alpha}\phi^{[n]}(a)|^2}{1-|\alpha|^2}-1.\\
	&=\dfrac{2|\alpha|^2-2Re{\Big(\bar{\alpha}\phi^{[n]}(a)\Big)}}{1-|\alpha|^2}\\
	&=\dfrac{2\Big(|\alpha|^2-Re{\Big({\bar{\alpha}}\phi^{[n]}(a)\Big)}\Big)}{1-|\alpha|^2}\\
	&=\dfrac{2\Big(|\alpha|^2-Re(\alpha)-Re{\Big({\bar{\alpha}}(\phi^{[n]}(a)-1)\Big)}\Big)}{1-|\alpha|^2}\\
	&=\dfrac{-2}{1-|\alpha|^2}\Big(Re{\Big({\bar{\alpha}}(\phi^{[n]}(a)-1)\Big)\Big)}\\
\end{split}
\end{equation*}
The last step came from the fact that $|\alpha|^2-Re(\alpha)=r^2-r\cos t=r(r-\cos t)=0$. Here, 
\begin{equation*}
\begin{split}
	2Re{\Big({\bar{\alpha}}(\phi^{[n]}(a)-1)\Big)}&=-2Re{\Big(\dfrac{{\bar{\alpha}}2i(1-a)}{2i+nb(1-a)}\Big)}\\
	&=\dfrac{{-\bar{\alpha}}2i(1-a)}{2i+nb(1-a)}+\dfrac{{\alpha}2i(1-\bar{a})}{-2i+nb(1-\bar{a})}\\
	&=\dfrac{-4(\bar{\alpha}(1-a)-(\overline{1-a})\alpha)-2i(\bar{\alpha}-\alpha)nb|1-a|^2}{4+2inb(a-\bar{a})+n^2b^2|1-a|^2}
\end{split}
\end{equation*}
So that, \begin{equation*}
\begin{split}
	\lim\limits_{n\rightarrow \infty}n\Big(\frac{a_n}{a_{n+1}}-1\Big)&=\dfrac{2}{1-|\alpha|^2}\dfrac{(\bar{\alpha}-\alpha)i}{b}\\
	&=\dfrac{-4\cos t\sin^{2} t}{-2\sin^{2} t\cos t}\\
	&=2>1.
\end{split}
\end{equation*}
The last step follows from the fact that $r=\cos t$, so that $\alpha=\cos^2 t+i\cos t\sin t$, and thus $b=Re(b)=Re(\frac{-2i\alpha}{1-\alpha})=\frac{2\cos t}{\sin t}.$
Hence, by Rabbe's test, the series $\sum_{n\in \mathbb{N}}a_n $ is finite. 

Since $\phi^{-1}$  is also a parabolic automorphism fixing $1$ by similar arguments as above and using the fact that $(\phi^{-1})^{'}(w)=\frac{1}{\phi^{'}(\phi^{-1}(w))}$ for any $w \in \mathbb{T}$, the series $\sum_{n\in \mathbb{N}}a_{-n}$ converges. Hence the series $\sum_{n\in \mathbb{Z}}a_n$  converges.

\end{proof}

	Matache \cite{matachenum} provided the similar examples for discrete singular inner eigenfunctions of composition operators induced by hyperbolic automorphism fixing $1$ and $-1$, in terms of the Poisson kernel. In the above theorem, we have generalized it to case of all the non-elliptic automorphism.

The previous theorem does not give any information about the corresponding eigenvalue. We now present some observations concerning the eigenvalues corresponding to singular inner eigenfunctions for composition operators induced by automorphisms.
\begin{remark}\label{value}
Suppose $S_\mu \circ \phi=\lambda S_\mu$. Let $\alpha:=\phi^{-1}(0)$. Since $S_\mu(0)=e^{-\mu(\mathbb{T})}>0$, then we have
	$$1=\Big|\frac{S_\mu\circ \phi(\alpha)}{S_\mu(\alpha)}\Big|=\Big|\frac{S_\mu(0)}{S_\mu(\alpha)}\Big|=\frac{S_\mu(0)}{|S_\mu(\alpha)|}.$$
	Consequently, $$\lambda=\frac{S_\mu(0)}{S_\mu(\alpha)}=\frac{|S_\mu(\alpha)|}{S_\mu(\alpha)}.$$
\end{remark}

\begin{remark}\label{discreteiegen}
For any automorphism $\phi$, excluding irrational elliptic automorphism case, the set  of all eigenvalues of $C_\phi$ with singular inner functions (with the atom set of the measure as a single orbit) as the eigenfunctions will be either singleton or the whole unit circle. To see this, 
take $\alpha:=\phi^{-1}(0)$ and let  $a\in \mathbb{T}$ be arbitrary. For a fixed positive real number $a_0$, consider the {singular inner function $S_\mu$ defined as in Theorem \ref{disingular}}. Then, we can see that $S_\mu(\alpha)= e^{a_0(p+iq)}$, where $p+iq=\sum_{n \in\mathbb{Z}}b_n\frac{\alpha+\phi^{[n]}(a)}{\alpha-\phi^{[n]}(a)} \neq 0$. For any $n \in \mathbb{Z}$, $b_n $ is defined such that  $a_n=a_0b_n$, so that $b_n $ is independent of the value  $a_0$. Hence, the value of $q$ solely depends upon $a$ and $\alpha$. Thus, the eigenvalue corresponding to $S_{\mu}$ is  $\frac{|S_\mu(\alpha)|}{S_\mu(\alpha)}=e^{-iqa_0}$.
	\begin{enumerate}
		\item  If $q$ is zero, then $\lambda=1$ is the only eigenvalue with eigenfunction of the form given as above.
		\item If $q$ is non-zero, then for any unimodular constant $\lambda$, there exist an eigenfunction induced by the measure of the form given in Theorem \ref{disingular}. 
		Indeed, for a given $\lambda=e^{i\theta }\in \mathbb{T}$ with $\theta \in (0,2\pi]$, choose $$a_0=\begin{cases}
			\frac{-\theta}{q}&\text{ if } q<0\\
			\frac{-\theta+4\pi}{q}&\text{ if } q>0.\\
		\end{cases}$$ Then, by Remark \ref{value}
		and Theorem \ref{disingular}, we have $S_\mu \circ \phi=\lambda S_\mu$.
	\end{enumerate}
	
\end{remark}

In \cite[Page~246, 247]{matachenum}, it have been given that, for $C_\phi$ where $\phi$ is a hyperbolic automorphism fixing $\pm1$, the eigenvalue corresponding to the singular inner eigenfunction, where the atom set of the measure is a single orbit, cannot be 1. Here we provide a correction for the statement, that is, we prove that 1 is also a possible eigenvalue for the eigenfunction under consideration.
 For a hyperbolic automorphism fixing $p$, for some $p\in \mathbb{T}$, the following theorem guarantees that the value $q$ given in above remark does not vanish at least for some $a \in \mathbb{T}$.
\begin{theorem}
	Let $\phi$ be any hyperbolic automorphism  fixing $\pm p$, for some $p\in \mathbb{T}$. For any $\lambda \in \mathbb{T}$, there exists a singular inner  eigenfunction  $S_\mu$ for $C_\phi$, where $\mu$ is a discrete measure with atom set as a single orbit.
\end{theorem}
\begin{proof}
Without loss of generality, let $p=1$  and $1$ is the Denjoy-Wolff point of $\phi$.	Using   \cite[Lemma 5.4.8]{texthardy}, we can compute that
	$$\phi^{[n]}(z)=\begin{cases}
		
		\frac{(d^n+1)z+(d^{n}-1)}{(d^n+1)+(d^{n}-1)z}&\text{ if } n\geq 1 \\
		\frac{(d^n+1)z-(d^{n}-1)}{(d^n+1)-(d^{n}-1)z}& \text{ if } n \leq -1,
		
	\end{cases}
	$$ for some $0<d<1$.
	Consider $a=i$ and for $n \in \mathbb{Z}$, let $b_n$ be defined as mentioned in Remark \ref{discreteiegen}.
	Then we have,
	$$q=Im\Big(\sum\limits_{n \in\mathbb{Z}}b_n\frac{\alpha+\phi^{[n]}(i)}{\alpha-\phi^{[n]}(i)}\Big)=-4\alpha\sum\limits_{n \in\mathbb{Z}}\frac{b_n d^n}{(d^{2n}+1)|\alpha-\phi^{[n]}(i)|^2}\neq 0,$$
	since each summand is strictly positive. By the Remark \ref{discreteiegen}, there exists a singular inner  eigenfunction  $S_\mu$ for $C_\phi$, where $\mu$ is a discrete measure with atom set as a single orbit.
\end{proof}

	   For an arbitrary hyperbolic automorphism, the set of eigenvalues corresponding to the discrete singular inner eigenfunction is the whole unit circle. This can be seen using the fact that any arbitrary hyperbolic automorphism can be turned into one which fixes $\pm 1$, via conjugation.
	   
	   The question regarding the singular inner eigenfunctions of $C_\phi$, arising from a singular continuous measure is still open.

		We conclude the section by providing a result regarding the Blaschke eigenfunction of $C_\phi$, where $\phi$ is a non-elliptic automorphism of $\mathbb{D}$.
	By \cite[Theorem 3.5]{matacheeigen} as well as \cite[Theorem 7.1]{modelsubspace}, any Blaschke eigenfunction will have its zero set as a disjoint union (possibly infinite)  of distinct orbits. In case of a hyperbolic automorphism $\phi$, Matache \cite[Remark 2]{matacheeigen} has given a class of  concrete examples for the Blaschke product  $B$, where  the zero set of $B$, denoted as $Z(B)$, is a countably infinite union of distinct orbits of $\phi$. Now, we establish the existence of such a Blaschke product for a general non-elliptic automorphism.
		\begin{proposition}
			Let $\phi$ be a non-elliptic automorphism of $\mathbb{D}$.  Then there exist Blaschke product B such that $Z(B)$ is an infinite union of distinct orbits and B is an
			eigenfunction of $C_\phi$.
		\end{proposition}
		\begin{proof} 
			Let $\phi$ be a parabolic automorphism of $\mathbb{D}$ . Without loss of generality, let $\phi$ fixes $1$. Let $\Gamma(z)=\frac{1+z}{1-z}$, for all $z \in \mathbb{D}$. Then there exist a non-zero real number $\beta$ such that $\Gamma \circ \phi \circ \Gamma^{-1}(s)=s+i\beta$ for all $s$ in the right half plane. 
			
			From the proof of the \cite[Theorem 4.1]{buerlingtype}, for an arbitrary $z \in \mathbb{D}$, we have 
			$$1-|\phi^{[m]}(z)|^2=\dfrac{4 u}{(m\beta+v)^2+(1+u)^2},$$
			for $m \in \mathbb{Z}$, where 
			$u$ and $v$ are the real and imaginary part of $\Gamma(z)$ respectively. For any $r \in [0,1)$, 
			$$1-r^2\leq 2(1-r)\leq 2(1-r^2).$$ Thus, for any sequence $\{z_n\}$ in $\mathbb{D}$, we have $\sum\limits_{n\in \mathbb{N}}(1-|z_n|)<\infty$ if and only if $\sum\limits_{n\in \mathbb{N}}(1-|z_n|^2)<\infty$. Thus, we can say that $\{\phi^{[m]}(z_n)\}_{m\in \mathbb{Z}, n \in \mathbb{N}}$ is a Blaschke sequence if and only if $$\sum\limits_{m\in \mathbb{Z}}\sum\limits_{n\in \mathbb{N}}\dfrac{4 u_n k_n}{(m\beta+v_n)^2 (1+u_n)^2}<\infty,$$
			where for each $n \in \mathbb{N}$, $k_n\in  \mathbb{N}$ and $u_n +iv_n=\Gamma(z_n)$ for $z_n\in \mathbb{D}$. Thus, choosing $z_n$ in such a way will provide us with a class of required examples. In particular, 
			choose $\{u_n+iv_n\}\subset\mathbb{H}^+$ such that $\sum u_n<\infty$ and $v_n\equiv d$, where $d$ is not a multiple of $\beta$. Let $z_n=\Gamma^{-1}(u_n+iv_n)$. Then we have
			$$\dfrac{4 u_n}{(m\beta+v_n)^2+(1+u_n)^2}<\dfrac{4 u_n}{(m\beta+d)^2}.$$ 
			
			For hyperbolic automorphism $\phi$, one can refer to 
			 \cite[Remark 2]{matacheeigen}.
		\end{proof}

\section{ Beurling and model invariant subspaces of composition operators}
In this section, we establish a relation between Beurling and model invariant subspaces of composition operators $ C_\phi$ on $H^2$, when the symbol $\phi$ is an automorphism. To obtain this connection, we begin with a result by Nordgren et al. \cite{COMP}, which provides a characterization of the common invariant subspaces of $C_\phi$ and $M_z^*$.

 If a subspace $M$ of a Hilbert space $H$ is invariant under a bounded linear operator $T$ of $H$, then we say that $M\in Lat(T)$.
We rewrite \cite[Theorem 4.1]{COMP} in the following manner.

\begin{theorem}\cite[Theorem 4.1]{COMP}\label{inv}
	Let $\phi$ be an inner function such that $\phi(0)\neq 0$. Then  
	$(z\theta H^2)^\perp\in Lat(C_\phi)$ if and only if 
 $\frac{\theta}{\theta\circ \phi}\in H^\infty$.

\end{theorem}

The above result yields the following observation.

\begin{theorem}\label{main1}
	Let $\phi$ be an automorphism of $\mathbb{D}$ such that $\phi(0)\neq0$ and let $\theta$ be an arbitrary inner function. Then the following statements are equivalent:
	
	\begin{enumerate}
		\item   $(z\theta H^2)^\perp \in Lat(C_{\phi}).$ 
		\item   $\theta H^2 \in Lat(C_{\phi^{-1}}).$
	
		\item   $(\theta \circ \phi^{[n]}) H^2 \in Lat(C_{\phi^{-1}})\text{ for some } n \in \mathbb{Z}.$
		\item   $(\theta \circ \phi^{[n]}) H^2 \in Lat(C_{\phi^{-1}})\text{ for all } n \in \mathbb{Z}.$
		\item   $(z(\theta\circ\phi^{[n]}) H^2)^\perp \in Lat(C_{\phi})\text{ for some } n \in \mathbb{Z}.$
		\item   $(z(\theta\circ\phi^{[n]}) H^2)^\perp \in Lat(C_{\phi})\text{ for all } n \in \mathbb{Z}.$
\end{enumerate}                                                 \end{theorem}

\begin{proof}

Let $\phi$ be an automorphism such that $\phi(0)\neq0$ and let $\theta$ be an inner function. Suppose that  $({z\theta H^2})^\perp \in Lat (C_\phi)$. This is equivalent to the equation $\theta =(\theta \circ \phi) g$ for some $g \in H^\infty$, which can be seen from Theorem \ref{inv}. By taking composition  of $\phi^{-1}$ on both sides of the equation, we have the equivalent condition $\frac{\theta  \circ \phi^{-1} }{\theta}\in H^\infty$. By \cite[Theorem 2.3]{buerlingtype}, this can be rewritten as 
$ \theta  H^2 \in Lat(C_{\phi^{-1}})$. This proves the equivalence of (1) and (2).

We observe that 	$\theta H^2 \in Lat (C_{\phi^{-1}})$ is equivalent to $\frac{\theta  \circ \phi^{-1} }{\theta}\in H^\infty$, which can in turn be reformulated as $\frac{\theta  }{\theta  \circ \phi}\in H^\infty$ and subsequently rewritten as $ (\theta \circ \phi) H^2 \in Lat(C_{\phi^{-1}})$.
By using the above equivalence iteratively, we easily obtain the equivalence of (2),(3) and (4). Again with the repetitive use of  \cite[Theorem 2.3]{buerlingtype}, one can get the equivalence of (1), (5) and (6).

\end{proof}

As an immediate consequence  of  Theorem \ref{main1} and \cite[Theorem 4.1]{javad}, we have the following result.
\begin{proposition}\label{beurling}
	Let $\phi$ be a non-identity, non-elliptic automorphism of $\mathbb{D}$. Then for any nonconstant inner function $\theta$, $C_{\phi}(\theta H^2) \subseteq\theta H^2$ if and only if either of the following holds:
	\begin{enumerate}
		\item $\theta$ is an eigenfunction of $C_\phi$.  
		\item $\theta= \psi \prod\limits_{n=0}^{\infty}\omega\circ\phi^{[-n]}$
	\end{enumerate}
	where $\psi$ is an (inner) eigenfunction of $C_\phi$ corresponding to the eigenvalue 1, and $\omega $ is a nonconstant inner function such that the product is pointwise convergent for all $z \in \mathbb{D}$.
\end{proposition}

The occurrence of (1) in  Theorem \ref{beurling} is discussed in Section 3. For the case (2), we need to know what are all the inner functions $\omega$ such that  the product $\prod_{n=0}^{\infty}\omega(\phi^{[-n]}(z))$ converges  pointwise for all $z \in \mathbb{D}$. In the case of hyperbolic automorphism $\phi$, \cite[Lemma 3.1]{javad} established a sufficient condition for the inner function $\omega$ such that the product $\prod_{n=0}^{\infty}\omega(\phi^{[-n]}(z))$ converges for all $z \in \mathbb{D}$. For the general case the question is still open.

Now, we discuss some results that arise as a consequence of the Theorem \ref{main1}.
For an inner function $\theta$, we denote $K_{\theta}=(z\theta H^2)^\perp$.
\begin{remark}
	Let $\theta $ be an inner function with only a finite number of zeros in $\mathbb{D}$. It follows from \cite[Theorem 5.17]{article1} and Theorem \ref{main1} that $
	K_{\theta}\notin Lat(C_\phi)$ for any non-elliptic automorphism $\phi$.
\end{remark}

In the special case where $\theta$ is a Blaschke product, the invariance of $K_\theta$ under $C_\phi$ can be characterized as follows.

\begin{proposition}\label{bl}
	Let $\phi$ be an automorphism such that $\phi(0)\neq 0$. For a Blaschke product $B$, $C_\phi(K_B)\subseteq K_B$ if and only if $ {mult}_{B}(w)\geq {mult}_{B\circ \phi}(w)$ for all $w\in Z(B)$.                  
\end{proposition} 

\begin{proof}
	Let $C_\phi(K_B)\subseteq K_B$. Then $\frac{B}{B\circ\phi} \in H^{\infty}$. Thus, we see that $ {mult}_{B}(w)\geq {mult}_{B\circ \phi}(w)$ for all $w\in Z(B)$. 
	
	Conversely, let  $ {mult}_{B}(w)\geq {mult}_{B\circ \phi}(w)$ for all $w\in Z(B)$. Using \cite[Proposition 3.7]{article1}, $B_2:=B\circ\phi$ is a Blaschke product. Thus, by the Reisz factorization theorem \cite[Theorem 2.5]{duren}, $B=B_2B_3$, where $B_3$ is a Blaschke product. Therefore, the ratio $\frac{B}{B\circ \phi}=\frac{B_2B_3}{B_2}=B_3 \in H^\infty$. Hence, by Theorem \ref{inv}, $C_\phi(K_B)\subseteq K_B$.
\end{proof}
 In view of \cite[Proposition 3.7]{article1}, using the similar proof technique of Proposition \ref{bl}, we have the following corollary.
\begin{corollary}
	Let $\phi$ be an inner function with $\phi(0)\neq 0$. Then for any finite Blaschke product $B$, $C_\phi(K_B)\subseteq K_B$ if and only if $ {mult}_{B}(w)\geq {mult}_{B\circ \phi}(w)$ for all $w\in Z(B)$.                  
\end{corollary}

The following proposition is the model-subspace analogue of Corollary \ref{corsplit}. It is obtained by combining Theorem \ref{main1} with the corresponding characterization, and is considerably more subtle and nontrivial.

\begin{corollary}\label{split2}
	Let $\phi$ be an automorphism of $\mathbb{D}$ and let $ \theta=BS_{\mu_a}S_{\mu_c}$ be its inner
	factorization.  Then
	$(\theta H^2)^\perp \in Lat( C_\phi)$  if and only if  $(zS_{\mu_a} H^2)^\perp \in Lat( C_\phi)$,
	$(zS_{\mu_c} H^2)^\perp \in Lat( C_\phi)$ and $(B H^2)^\perp \in Lat( C_\phi)$.
\end{corollary}

We now present some results concerning the singular Beurling subspaces that are invariant under $C_\phi$, where $\phi$ is a non-elliptic automorphism.

The following equivalent condition for the invariance of a singular model subspace under composition operators induced by a non-elliptic automorphism is obtained from \cite[Corollary 2.15]{valentine} and Theorem \ref{main1}.
\begin{theorem}\label{pure}
	Let $\phi$ be a non-elliptic automorphism of $\mathbb{D}$ and $\mu$ be a discrete singular measure on $\mathbb{T}$. Then $C_{\phi}(zS_{\mu}H^2)^{\perp}\subseteq(zS_{\mu}H^2)^{\perp}$
	if and only if 
	\begin{enumerate}
		\item $\phi^{-1}(atom(\mu))\subseteq{atom(\mu)}$.
		\item $\dfrac{\mu\{\phi^{-1}(w)\}}{\mu\{w\}}\geq |({\phi^{-1}})^{'}(w)|$ for all $w \in atom{(\mu)}$.
	\end{enumerate}
\end{theorem}
Using the above proposition, we now make some observations concerning the singular inner functions with finitely many atoms.

\begin{corollary}
	Let $S(z)=\exp(\alpha\frac{z+p}{z-p})$ be an atomic singular inner function, where  $\alpha>0$ and $p \in \mathbb{T}$.
	\begin{enumerate}
		\item	Let $\phi$  be a parabolic automorphism. Then $C_\phi(zSH^2)^\perp\subseteq (zSH^2)^\perp$ if and only if
		$p$ is the fixed point of $\phi$.
		\item  Let $\phi$  be a hyperbolic automorphism. Then $C_\phi(zSH^2)^\perp\subseteq (zSH^2)^\perp$ if and only if
		$p$ the fixed point of $\phi$ and $p$ is not the Denjoy-Wolff point.
\end{enumerate}
	
\end{corollary} 
\begin{proof}
	
Proof of (1):	Let $\phi$ be a parabolic automorphism.
	If $C_\phi(zSH^2)^\perp\subseteq (zSH^2)^\perp$ , then by the first condition of Theorem \ref{pure}, we have $p \in \{\phi(p)\}$ and thus $p$ is the fixed point of $\phi$. The converse part directly follows from Theorem \ref{pure}.
	 
Proof of (2):	Let $\phi$ be a hyperbolic automorphism. Assume that $C_\phi(zSH^2)^\perp\subseteq (zSH^2)^\perp$. Then condition (1) in Theorem \ref{pure} implies $\phi(p)=p$. We can see that, for any $w \in \mathbb{T}$, $(\phi^{-1})^{'}(w)=\frac{1}{\phi^{'}(\phi^{-1}(w))}$. By condition (2) in Theorem \ref{pure}, we have 
$|\phi^{'}(p)|\geq 1$, which is only possible if $p$ is not the  Denjoy–Wolff point. The converse part directly follows from Theorem \ref{pure}.

\end{proof}

\begin{proposition}
		Let $\mu$ be  a discrete measure with exactly 2 atoms and let $\theta=BS_\mu$ be the inner factorization of the function $\theta$. Then $\theta H^2$ and $(z\theta H^2)^\perp$ are not in $Lat (C_\phi)$ for any parabolic automorphism $\phi$.
\end{proposition}
\begin{proof}

	Let $atom(\mu)=\{a,b\}$, and let $\phi$ be any parabolic automorphism.  Suppose $\theta H^2=BS_\mu H^2\in Lat (C_\phi)$, which  implies $S_\mu H^2 \in Lat (C_\phi)$. Then by  \cite[Corollary 2.15]{valentine}, $\phi\{a,b\}\subset\{a,b\}$. Since $\phi$ is an automorphism, either $\phi$ fixes $a$ and $b$ or $\phi$ swaps $a$ and $b$. In either cases, $\phi$ or $\phi^2$ fixes $a$ and $b$, which is not possible for any parabolic automorphism.  
	
	Similarly, assume that $(z\theta H^2)^\perp\in Lat (C_\phi)$. Then by Corollary \ref{split2}, we get 
	$(zS_\mu H^2)^\perp\in Lat (C_\phi)$ and by Proposition \ref{pure}, we get  $\phi^{(-1)}\{a,b\}\subset\{a,b\}$. Using  similar argument as above, we will get a contradiction.
	
\end{proof}
\begin{proposition}
	Let $\mu$ be  a discrete measure  whose atom set is finite with at least 3 elements and let $\theta=BS_\mu$. Then $\theta H^2$ and $(z\theta H^2)^\perp$ are not in $Lat (C_\phi)$ for any non-elliptic automorphism $\phi$.
\end{proposition}                                                  \begin{proof}
	Let $\theta=BS_\mu$ be an inner function such that $atom(\mu)=\{a_1, a_2, ,…,a_n\}$ for some $a_i \in \mathbb{T}$ and $n \in \mathbb{N}$ and let $\phi$ be a non-elliptic automorphism of $\mathbb{D}$. Assume that, $\theta H^2\in Lat (C_\phi)$. Since  $\phi$ is bijective on $atom(\mu)$, by \cite[Corollary 2.15]{valentine}, we have $\phi(atom(\mu))=atom(\mu)$. Thus, $\phi^{[n]}\in Aut(\mathbb{D})$ fixes $\{a_1, a_2, ,…,a_n\}$ and for $n\geq3$, this contradicts the fact that $\phi$ is a non-elliptic automorphism. Hence, $\theta H^2\notin Lat (C_\phi)$. 
	
	Similarly, assume that $(z\theta H^2)^\perp\in Lat (C_\phi)$. Then by Corollary \ref{split2}, we get 
	$(zS_\mu H^2)^\perp\in Lat (C_\phi)$ and by Proposition \ref{pure},  we get
	$\phi^{(-1)}(atom(\mu))= (atom(\mu))$. Using similar argument as above, we get a contradiction. Hence, $(z\theta H^2)^\perp \notin Lat (C_\phi)$.
\end{proof}
We conclude the article with a result concerning the minimal invariant subspaces of the composition operator.
Let $T$ be a bounded linear operator on a Hilbert space $H$. A subspace $M \in Lat(T)$ is said to be minimal invariant if there does not exist a subspace $N \in Lat(T)$ such that ${0}\neq N\subsetneq M$.

For $a \in \mathbb{D}$, let $\phi_a$ be a holomorphic self-map of $\mathbb{D}$ defined as $\phi_a(z)=az+1-a$ for $z \in \mathbb{D}$. Carmo and Noor \cite[Corollary 16]{noor} proved that no Beurling subspace invariant under $C_{\phi_a}$ is minimal. Here we give a simpler proof for the statement and extend it to an arbitrary holomorphic self-map of $\mathbb{D}$.

\begin{proposition}
	Let $\phi$ be a holomorphic self-map of $\mathbb{D}$ and let $\theta H^2 \in Lat(C_\phi)$. Then $\theta^nH^2\in Lat(C_{\phi})$ for $n \in \mathbb{N}$. Hence, none of the Beurling subspaces is minimal invariant for the composition operator $C_\phi$.
\end{proposition}
\begin{proof}
	Let $\phi$  be a holomorphic self-map of $\mathbb{D}$ and $\theta H^2 \in Lat(C_\phi)$ for some inner function $\theta$. Then $f=\frac{\theta^n \circ \phi}{\theta^n} =\Big(\frac{\theta \circ \phi}{\theta}\Big)^n \in H^\infty$ for $n \in \mathbb{N}$. Thus, by \cite[Theorem 2.3]{buerlingtype}, $f  \in H^\infty$ and so that $\theta^nH^2\in Lat(C_{\phi})$. We can see that $\theta^{n+1}H^2\subsetneq \theta^{n}H^2$ for $n \in \mathbb{N}$. Hence, 
	$\theta H^2$ is not a minimal invariant subspace of $C_\phi$.
	
\end{proof}

{\bf Data Availability.} The authors declare that this research is purely theoretical
and does not associated with any data.

{\bf Acknowledgments.} The first author is supported by the DST WISE Fellowship for Ph.D. (WISE-PhD)(DST/WISE-PhD/PM/2024/113).
 \nocite{*} 
\bibliographystyle{amsplain}

\end{document}